\documentclass{elsarticle}

\usepackage{fullpage}
\makeatletter
\def\ps@pprintTitle{%
	\let\@oddhead\@empty
	\let\@evenhead\@empty
	\let\@oddfoot\@empty
	\let\@evenfoot\@oddfoot
}
\makeatother

\usepackage{amssymb}
\usepackage{float}
\usepackage{textcomp}
\usepackage{amsmath}
\usepackage{amsfonts}
\usepackage{amsthm}
\usepackage{latexsym}
\usepackage{tikz}
\usepackage{hyperref}
\usepackage{url}
\usetikzlibrary{arrows,%
	petri,%
	topaths}%
\usepackage{tkz-berge}
\usepackage[position=top]{subfig}
\usepackage[linesnumbered,ruled,noend]{algorithm2e}
\newtheorem{lemma}{Lemma}
\newtheorem*{theorem*}{Theorem}
\usepackage[symbol]{footmisc}
\usepackage{mathtools}

\newcommand{\source}{\mathfrak{s}}

\newcommand{\targets}{\mathfrak{T}}
\newcommand{\target}{\mathsf{t}}

\newcommand{\Prob}{\texttt{Pr}}
\newcommand{\prob}{\mathbb{P}}

\newcommand{\rel}{\mathfrak{Rel}}
\newcommand{\bags}{\mathfrak{B}}

\newcommand{\problem}[1]{\textsc{#1}}

\begin{document}

\begin{frontmatter}

\title{An Efficient Algorithm for Computing\\ Network Reliability in Small Treewidth}

\author{Amir Kafshdar Goharshady}
\address{IST Austria (Institute of Science and Technology Austria)\\Am Campus 1, 3400 Klosterneuburg, Austria\\ amir.goharshady@ist.ac.at}

\author{Fatemeh Mohammadi}
\address{School of Mathematics, University of Bristol\\University Walk, Bristol, BS8 1TW, United Kingdom\\
fatemeh.mohammadi@bristol.ac.uk}

\begin{abstract}
	We consider the classic problem of \problem{Network Reliability}. A network is given together with a source vertex, one or more target vertices, and probabilities assigned to each of the edges. Each edge appears in the network with its associated probability and the problem is to determine the probability of having at least one source-to-target path. This problem is known to be NP-hard.
	
	We present a linear-time fixed-parameter algorithm based on a parameter called treewidth, which is a measure of tree-likeness of graphs. \problem{Network Reliability} was already known to be solvable in polynomial time for bounded treewidth, but there were no concrete algorithms and the known methods used complicated structures and were not easy to implement. We provide a significantly simpler and more intuitive algorithm that is much easier to implement.
	
	We also report on an implementation of our algorithm and establish the applicability of our approach by providing experimental results on the graphs of subway and transit systems of several major cities, such as London and Tokyo. To the best of our knowledge, this is the first exact algorithm for \problem{Network Reliability} that can scale to handle real-world instances of the problem.  
\end{abstract}


\end{frontmatter}

\section{Introduction}

\noindent\textbf{\problem{Network Reliability}.} Consider a network modeled as a graph $G = (V, E)$, where each edge $e \in E$ has a known probability of failure. For example, the graph might be a model of communication links in a mobile network or railway lines between subway stations. Given a source vertex $\source$ and a set $\targets$ of target vertices, the goal of the \problem{Network Reliability} problem is to assess the reliability of connections between $\source$ and $\targets$. Concretely, the \problem{Network Reliability} problem asks for the probability of existence of at least one source-to-target path that does not pass through failed edges. In the examples mentioned above, this is equivalent to asking for the probability of being able to send a message from $\source$ to $\targets$ through the mobile network or the probability of being able to travel from $\source$ to $\targets$ in the subway network.

\smallskip\noindent\textbf{Short History.} Network reliability is an important and well-studied problem with surveys appearing as early as 1983~\cite{reliabilitySurvey}. Aside from the obvious applications such as the two mentioned above, the problem has many other surprising applications, including analysis and elimination of redundancy in electronic systems and electrical power networks~\cite{reliabilityApplication}. 
\problem{Network Reliability} was shown to be NP-hard for general graphs~\cite{ball1986computational} and hence researchers turned to solving it in special cases~\cite{ball1986computational,reliabilitySurvey}, such as series-parallel graphs~\cite{satyanarayana1985linear} and graphs with limited number of cuts~\cite{provan1984computing}. Genetic~\cite{coit1996reliability}, randomized~\cite{karger2001randomized}, approximate~\cite{srivaree2002estimation} and Monte Carlo~\cite{gertsbakh2016models} algorithms are studied extensively as well. There are also several algebraic studies of the problem with the goal of obtaining bounds in series-parallel and other special families of graphs~\cite{brown1996cohen,fm2016combinatorial,saenz2015hilbert,mohammadi2016algebraic,mohammadi2016divisors}. Several variants of the problem are defined~\cite{guidotti2017network,zhang2017game,yeh2015new}, and approaches to modify the network for its optimization are also investigated~\cite{gutjahr1996configurations}. In this paper, we consider a parameterization of \problem{Network Reliability} and obtain a linear-time algorithm. 

\smallskip\noindent\textbf{Parameterized Algorithms.} An efficient \emph{parameterized} algorithm solves an  optimization problem in polynomial time with respect to the size of input, but possibly with non-polynomial dependence on a specific aspect of the input's structure, which is called a ``parameter"~\cite{niedermeier2002invitation}. A problem that can be solved by an efficient parameterized algorithm is called \emph{fixed-parameter tractable} (FPT). For example, there is a polynomial-time algorithm for computing minimal cuts in graphs whose runtime is exponentially dependent on the size of the resulting cut~\cite{cygan2015parameterized}. Exploiting the additional benefit of having a parameter, parameterized complexity provides finer detail than classical complexity theory~\cite{downey2012parameterized}. 

\smallskip\noindent\textbf{Treewidth.} A well-studied parameter for graphs is the \emph{treewidth}, which is a measure of tree-likeness of graphs~\cite{robertson1990graph}. Many hard problems are shown to have efficient solutions when restricted to graphs with small treewidth~\cite{courcelle1990monadic,chatterjee2016algorithms,courcelle1993monadic,datapacking,bodlaender1988dynamic,chatterjee2018algorithms,chatterjee2016Jtdec}. Notably,~\cite{wolle2002framework} introduces a general framework that shows several variants of the network reliability problem can be solved in polynomial time when parameterized by the treewidth.
Many real-world graphs happen to have
small treewidth~\cite{bodlaender1994tourist,thorup1998all,gustedt2002treewidth,amirtr}. In this work, we show that subway and transit networks often have this property.

\smallskip\noindent\textbf{Our contribution.} Our contribution is providing a new fixed-parameter algorithm for finding the exact value of \problem{Network Reliability}, using treewidth as the parameter. Our algorithm, while being linear-time, is much shorter and simpler than the general framework utilized in~\cite{wolle2002framework}. We also provide an implementation of our algorithm and experimental results over the graphs of subway networks of several major cities. To the best of our knowledge, this is the first algorithm for finding the exact value of \problem{Network Reliability} that can scale to handle real-world instances, i.e.~subway and transit networks of major cities.

Graphs with constant treewidth are the most general family of networks for which exact algorithms for computing reliability are found. This family contains trees, series-parallel graphs and outerplanar graphs~\cite{bodlaender1994tourist}.

\smallskip\noindent\textbf{Structure of the Paper.} The present paper is organized as follows: First, Section~\ref{sec:reliability} provides formal definitions of the \problem{Network Reliability} problem and Treewidth. Then, Section~\ref{sec:algo}, which is the main part of the paper, presents our simple linear algorithm for solving \problem{Network Reliability} in graphs with constant treewidth. Section~\ref{sec:exp} contains a report of our implementation, which is publicly available, and establishes the applicability of our approach by providing experimental results on real-world subway networks.

\section{Preliminaries} \label{sec:reliability}

In this section, we formalize our notation, and define the problem of \textsc{Network Reliability} and the notion of treewidth.

\smallskip\noindent\textbf{Multigraphs.} A multigraph is a pair $G = (V, E)$ where $V$ is a finite set of vertices and $E$ is a finite multiset of edges, i.e.~each $e \in E$ is of the form $\{u, v\}$ for $u, v \in V.$ The vertices $u$ and $v$ are called the endpoints of $e$. Note that $E$ might contain distinct edges that have the same endpoints. In the sequel, we only consider multigraphs and simply call them graphs for brevity.

\smallskip\noindent\textbf{Notation.} Given a graph $G = (V, E)$, a \emph{path} from $u \in V$ to $w \in V$ is a finite sequence $u=u_0, u_1, \ldots, u_l=w$ of distinct vertices such that for every $i<l$, there exists an edge $e = \{u_i, u_{i+1}\} \in E.$ We write $u \leadsto_G w$ to denote the existence of a path from $u$ to $w$ in $G$. We simply write $u \leadsto w$ if $G$ can be deduced from the context. For a set $E' \subseteq E$, we write $u \leadsto_{E'} w$ if there exists a path from $u$ to $w$ whose every edge is in $E'$. A \emph{connected component} of a graph $G$ is a maximal subset  $C \subseteq V$ such that for every $c_1, c_2 \in C$, we have $c_1 \leadsto_G c_2.$ The graph $G$ is called \emph{connected} if it has exactly one connected component. A \emph{cycle} in the graph $G$ is a sequence $w_0, w_1, \ldots, w_l$ of vertices with $l>0$, such that for every $i<l$, there exists an edge $e = \{w_i, w_{i+1}\} \in E$ and all $w_i$ are distinct except that $w_0 = w_l$. A graph is called a \emph{forest} if it has no cycles. A forest is called a \emph{tree} if it is connected. In other words, a tree is a connected graph with no cycles.

\smallskip\noindent\textbf{Network Reliability Problem.}
A \problem{Network Reliability} problem \emph{instance} is a tuple $I = (G, \source, \targets, \Prob)$ where $G = (V, E)$ is a connected multigraph, $\source \in V$ is a ``source'' vertex and $\targets \subseteq V$ a set of ``target'' vertices. $\Prob$ is a function of the form $\Prob : E \rightarrow \left[0, 1\right]$ which assigns a probability to every edge of the graph $G$. The reliability problem on instance $I$ is then defined as follows: A new graph $G^s$ is probabilistically constructed such that its vertex set is $V$ and each edge $e \in E$ appears in it with probability $\Prob(e)$. Appearance of the edges are stochastic and independent of each other. The \problem{Network Reliability} problem asks for the probability $\rel(I)$ of having at least one path from the source vertex $\source$ to a target vertex $\target \in \targets$ in $G^s$.

\smallskip\noindent\textbf{Remark.} We are describing our approach on undirected graphs. However, it is straightforward to change all the steps of the algorithm to handle direcetd graphs as well.

We now provide a quick overview of the basics of tree decompositions and treewidth. A much more involved treatment can be found in~\cite{cygan2015parameterized,bodlaender1994tourist}.

\smallskip\noindent\textbf{Tree Decompositions.}
Given a connected multigraph $G = (V, E)$, a tree $T = (\bags, E_T)$  with vertex set $\bags$ and edge set $E_T$ is called a tree decomposition of $G$ if the following four conditions hold:
\begin{itemize}
	\item Each vertex $b \in \bags$ of the tree $T$ has an assigned set of vertices $V(b)\subseteq V.$ To distinguish vertices of $T$ and $G$, we call each vertex of $T$ a \emph{bag}. 
	\item Each vertex appears in some bag, i.e. $\bigcup_{b \in \bags} V(b)=V.$
	\item Each edge appears in some bag, i.e. $\forall e = \{u, v\} \in E~~\exists b \in \bags ~\mathrm{s.t.}~ \{u, v\} \subseteq V(b)$. We denote the set of edges that appear in a bag $b$ with $E(b)$. Note that an edge appears in $b$ if and only if both of its vertices do.
	\item Each vertex $v \in V$ appears in a connected subtree of $T$. More precisely, we let $\bags_v$ to be the set of bags whose vertex sets contain $v$, then $\bags_v$ must be a connected subtree of $T$.
\end{itemize}

\begin{figure}[h]
	\begin{center}
		\begin{minipage}{0.4\textwidth}
			\begin{flushright}
				\includegraphics[scale=0.4]{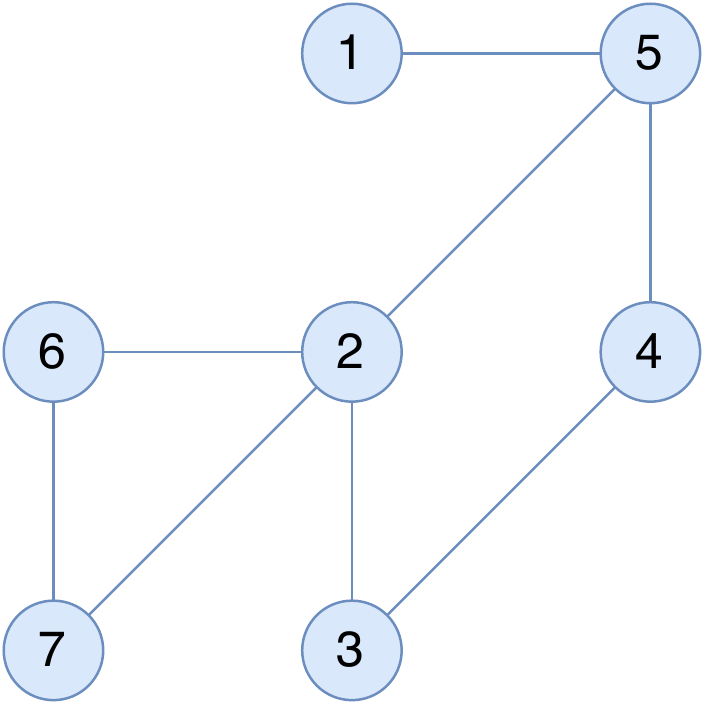}
			\end{flushright}
		\end{minipage}
		\begin{minipage}{0.4\textwidth}
			\begin{flushleft}
				\includegraphics[scale=0.4]{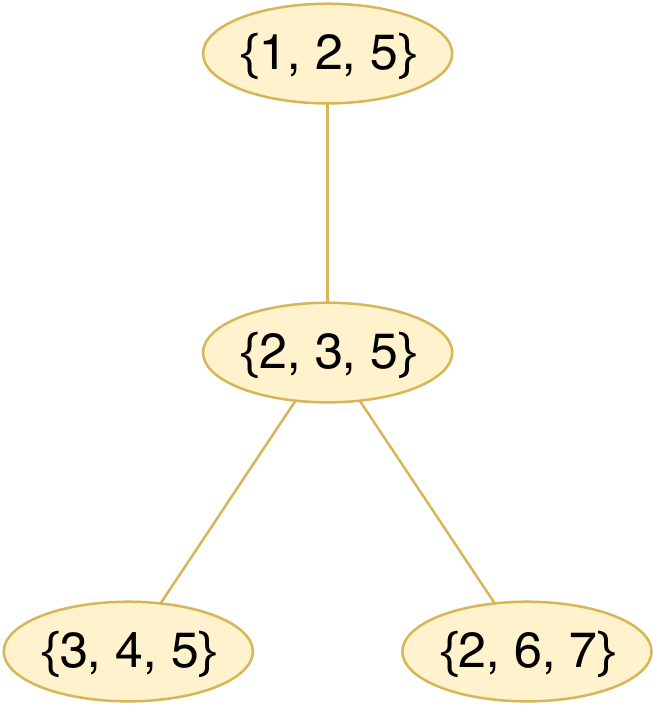}
			\end{flushleft}
		\end{minipage}
	\end{center}
	\caption{A graph (left) and a $2$-decomposition of it (right)}
	\label{fig:tw}
\end{figure}

See Figure~\ref{fig:tw} for an example.
A rooted tree decomposition is a tree decomposition in which a unique bag
is specified as ``root". Given two bags $b$ and $b'$, we say that $b$ is an \emph{ancestor} of $b'$ if $b$ appears in the unique path from the root to $b'$. In this case, we say that $b'$ is a \emph{descendant} of $b$. Note that each bag is both an ancestor and a descendant of itself. The bag $b$ is called the \emph{parent} of $b'$ if it is an ancestor of $b'$ and has an edge to $b'$, i.e.~$\{b, b'\} \in E_T$. In this case, we say that $b'$ is a \emph{child} of $b$. A bag $b$ with no children is called a \emph{leaf}.

\smallskip\noindent\textbf{Treewidth.} If a tree decomposition $T$ has bags of size at most $k+1$, then it is called a $k$-decomposition or a decomposition of \emph{width} $k$. The \emph{treewidth} of a graph $G$ is defined as the smallest $k$ for which a $k$-decomposition of $G$ exists. Intuitively, the treewidth of a graph measures how tree-like it is and graphs with smaller treewidth are more similar to trees.

\smallskip\noindent\textbf{Cut Property.} Tree decompositions are important for algorithm design because removing the vertices of each bag $b$ from the original graph $G$ cuts it into connected components corresponding to the subtrees formed in $T$ by removing $b$~\cite{bodlaender1994tourist}. We call this the ``cut property'' and it allows bottom-up dynamic programming algorithms to operate on tree decompositions almost the same way as in trees~\cite{bodlaender1988dynamic}. We use this property in our algorithm in Section~\ref{sec:algo}. We now formalize this point: 

\smallskip\noindent\textbf{Separators.} Given a graph $G = (V, E)$ and two sets of vertices $A, B \subseteq V$, we call the pair $(A, B)$ a separation of $G$ if (i)~$A \cup B = V$, and (ii)~no edge connects a vertex in $A \setminus B$ to a vertex in $B \setminus A$. We call $A \cap B$ the separator corresponding to the separation $(A, B)$.

\begin{lemma}[Cut Property~\cite{cygan2015parameterized}] \label{lemma:cut}
	Let $T = (\bags, E_T)$ be a tree decomposition of the graph $G$ and let $e = \{a, b\} \in E_T$ be an edge of $T$. By removing $e$, $T$ breaks into two connected components, $T^a$ and $T^b$, respectively containing $a$ and $b$. Let $A = \bigcup_{t \in T^a} V(t)$ and $B = \bigcup_{t \in T^b} V(t)$. Then $(A, B)$ is a separation of $G$ with separator $V(a) \cap V(b)$.
\end{lemma}

\begin{figure}[h]
	\begin{center}
		\begin{minipage}{0.4\textwidth}
			\begin{flushright}
				\includegraphics[scale=0.4]{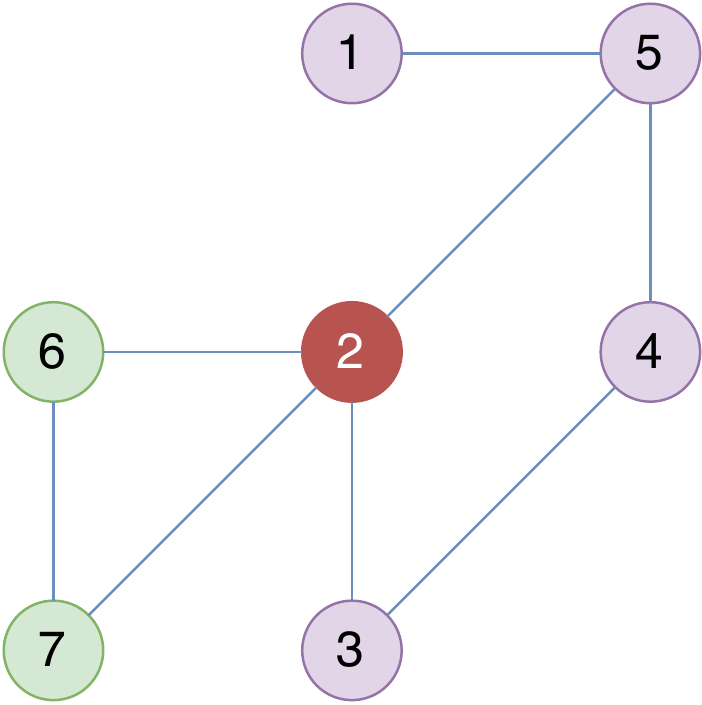}
			\end{flushright}
		\end{minipage}
		\begin{minipage}{0.4\textwidth}
			\begin{flushleft}
				\includegraphics[scale=0.4]{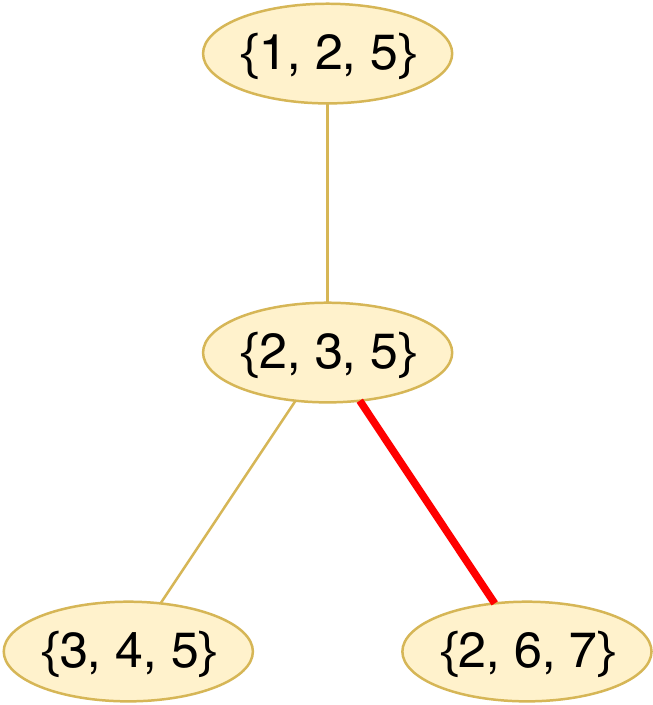}
			\end{flushleft}
		\end{minipage}
	\end{center}
	\caption{Cut Property}
	\label{fig:cut}
\end{figure}

For example, consider the same graph as in Figure~\ref{fig:tw}. By removing the edge from $\{2, 3, 5\}$ to $\{2, 6, 7\}$, the tree decomposition breaks into two connected components, containing the vertices $A = \{2, 6, 7 \}$ and $B = \{ 1, 2, 3, 4, 5 \}$, respectively. This is a separation of the graph with separator $\{2\}.$ These points are illustrated in Figure~\ref{fig:cut}.

\smallskip\noindent\textbf{Computing a Tree Decomposition.} In our algorithm in Section \ref{sec:algo}, when we operate on a graph $G$ with $n$ vertices and constant treewidth, we assume that we are also given a tree decomposition of $G$ as part of the input. This is justified by an algorithm of Bodlaender~\cite{bodlaender1996linear}, that given a graph $G$ and a constant $k$, decides in \emph{linear time} whether $G$ has treewidth at most $k$ and if so, produces a $k$-decomposition of $G$ with $O(n \cdot k)$ bags.

\section{Algorithm for \problem{Network Reliability}} \label{sec:algo}

In this section, we provide an algorithm for solving instances
of the \problem{Network Reliability} problem on graphs based on their tree decompositions.

\smallskip\noindent\textbf{Specification.} The
input to the algorithm is a \problem{Network Reliability} instance $I = (G, \source, \targets,  \Prob)$ together with a $k$-decomposition $T = (\bags, E_T)$ of the graph~$G$. The output is the
reliability $\rel(I)$, i.e.~the probability of existence of a path from $\source$ to $\targets$. Given that
the tree decomposition can be rooted at any bag, without loss of generality,
we assume that the source vertex $\source$ is in the root bag. We also assume that $G$ has $n$ vertices and $\vert \bags\vert \in O(n \cdot k)$. This can be obtained by an algorithm described in~\cite{bodlaender1996linear}.

\smallskip\noindent\textbf{Methodology.} Our algorithm is based on a technique called ``kernelization''~\cite{cygan2015parameterized}: Using the tree-decomposition $T$, we repeatedly shrink the graph $G$ to obtain smaller graphs that all have the same reliability as $G$. We continue our shrinking until we reach a graph that has very few vertices, i.e.~at most $O(k)$ vertices. We then use brute force to compute the reliability of this graph.

\smallskip\noindent\textbf{Discrete Probability Distributions.} Given a finite set $X$, a \emph{probability distribution} over $X$ is a function $\Prob: X \rightarrow [0, 1]$, that assigns a probability to each member of $X$, such that $\sum_{x\in X} \Prob(x) = 1.$

We first define an extension of the \problem{Network Reliability} problem, in which the probabilities of appearance of the edges need not be independent anymore, i.e.~some edges are \emph{correlated}. Although this extension makes the problem more general, it helps in finding a solution. As we will later see, it allows us to apply a shrinking procedure as described above.

\smallskip\noindent\textbf{Extended Network Reliability.} An \problem{Extended Network Reliability} instance with $r$ parts is a tuple $I = (G, E_1,  \ldots, E_r, \source, \targets, \Prob_1, \ldots, \Prob_r)$ in which:

\begin{itemize}
	\item $G = (V, E)$ is a connected graph;
	\item The $E_i$'s are pairwise disjoint multisets of edges and $\bigcup_{i=1}^r E_i = E;$
	\item $\source \in V$ is the source vertex;
	\item $\targets \subseteq V$ is the set of target vertices; and
	\item Each $\Prob_i: 2^{E_i} \rightarrow [0, 1]$ is a probability distribution over the subsets of $E_i$. 
	
\end{itemize}

We now define the \problem{Extended Network Reliability} problem on the instance $I$ as follows: a new graph $G^s$ is probabilistically constructed such that its vertex set is $V$ and its edge set is a subset  $E^s \subseteq E$ chosen probabilistically as follows:
\begin{itemize}
	\item For every part $E_i$, a subset $E^s_i \subseteq E_i$ of edges is probabilistically chosen according to the distribution $\Prob_i$. The $E^s_i$'s are chosen independently of each other.
	\item The set $E^s$ is defined as $E^s = \bigcup_{i=1}^r E^s_i.$
\end{itemize}
The \problem{Extended Network Reliability} problem asks for the probability $\rel(I)$ that the probabilistically-constructed graph $G^s$ contains a path from $\source$ to $\targets$. Intuitively, appearance of every edge in each part $E_i$ is correlated to every other edge in $E_i$, but independent of all the edges outside of $E_i$. A \problem{Network Reliability} instance is simply an \problem{Extended Network Reliability} instance in which each $E_i$ consists of a single edge, i.e.~every edge is independent of every other edge.

We now provide a simple brute force algorithm for the \problem{Extended Network Reliability} problem. This algorithm will later serve as a subprocedure in our main algorithm.

\smallskip\noindent\textbf{The Brute Force Algorithm.} Consider an \problem{Extended Network Reliability} instance $I$ as above and a graph $G' = (V, E')$ where $E' \subseteq E$, i.e.~a graph with the same set of vertices as $G$, but only a subset of its edges. We can easily compute $\prob(G^s = G'),$ i.e.~the probability that the probabilistically-constructed graph $G^s$ is equal to $G'$. We use each $\Prob_i$ to find the probability of the specific combination of correlated edges that are present in $E' \cap E_i$. Therefore, we have:
$$
\prob(G^s = G') = \prod_{i=1}^{r} \Prob_i(E' \cap E_i).
$$
Now $\rel(I)$ is simply the sum of $\prob(G^s = G')$ over those graphs $G'$ in which there is a path from $\source$ to $\targets$. Hence, we can use the brute force method as in Algorithm~\ref{alg:bf} for answering the \problem{Extended Network Reliability} problem. The algorithm creates all possible subgraphs $G'$ and checks if there is a path from $\source$ to $\targets$ in $G'$. If so, it computes the probability $\prob(G^s = G')$. Finally, it returns the sum of computed probabilities. 

\begin{algorithm}[H]
	\SetKwInOut{Input}{Input}\SetKwInOut{Output}{Output}
	\Input{An \problem{Extended Network Reliability} instance $I = (G, E_1, \ldots, E_r, \source, \targets, \Prob_1, \ldots, \Prob_r)$}
	\Output{$\rel(I)$}
	$\texttt{ans} \leftarrow 0$\;
	\ForEach{$E' \subseteq E$}
	{
		$G' \leftarrow (V, E')$\;
		\If{$\source \leadsto_{G'} \targets$}
		{
			$p \leftarrow 1$\;
			\For{$i \in \{1, \ldots, r\}$}
			{
				$p \leftarrow p \cdot \Prob_i(E' \cap E_i)$
			}
			$\texttt{ans} \leftarrow \texttt{ans} + p$\;
		}	
	}
	\textbf{return} $\texttt{ans}$\;
	
	\caption{The brute force method}\label{alg:bf}
\end{algorithm}

\smallskip\noindent\textbf{Complexity of the Brute Force Algorithm.} Assuming that the graph $G$ has $n$ vertices and $m$ edges, Algorithm~\ref{alg:bf} considers at most $2^{m}$ different cases for $E'$ (Line~2). In each case, checking reachability (Line~4) can be done in $O(m)$ using standard algorithms such as DFS or BFS, and computing $\prob(G^s = G')$, i.e.~the variable $p$ in the algorithm (Lines~5--7), also takes $O(m)$ time.  Hence, Algorithm~\ref{alg:bf} has a total runtime of $O(m \cdot 2^m)$ which is exponential. Therefore, this algorithm is only applicable to very small graphs.

In order to use Algorithm~\ref{alg:bf} on larger graphs, we need to shrink them to smaller graphs with the same reliability. The following lemmas are our main tools in doing so.

\begin{lemma} \label{lemma:avvali}
	Let $G_B = (B, E_B)$ be a graph, $\targets$ a set of target vertices, $B^* \subseteq B$ a subset of vertices that contains a target vertex $\target^* \in \targets$. Also, let $E^*$ be the set of all possible edges over $B^*,$ i.e.~$E^* = \{ \{u, v\} ~ \vert ~ u, v \in B^* \}.$ Then, there exists a function $f : 2^{E_B} \rightarrow 2^{E^*}$ that maps every subset $E'$ of edges of $E_B$ to a subset $f(E')$ of edges of $E^*$, such that:
	\begin{itemize}
		\item For all $a \in B^*$, we have $a \leadsto_{f(E')} \targets$ if and only if $a \leadsto_{E'} \targets.$
		\item For all $a, b \in B^*$ such that $a \not\leadsto_{E'} \targets$ and $b \not\leadsto_{E'} \targets$, we have $a \leadsto_{f(E')} b$ if and only if $a \leadsto_{E'} b.$
		\item For all $a, b \in B^*$ such that $a \leadsto_{E'} \targets$ and $b \leadsto_{E'} \targets$, we have $a \leadsto_{f(E')} b.$
	\end{itemize}  
	Moreover, given $E'$, one can compute $f(E')$ in linear time, i.e.~$O(\vert B \vert + \vert E_B \vert).$
\end{lemma}

Intuitively, the lemma above says that from a graph with vertex set $B$, one can create a smaller ``digest'' graph with vertex set $B^* \subseteq B$, in which (i)~a vertex has a path to a target if and only if it used to have a path to a target in the first place, (ii) any two vertices that do not have a path to a target are in the same connected component if and only if they used to be in the same connected component in the first place, and (iii) all the vertices that can reach a target are put in the same connected component. Indeed, the construction below merges all the vertices that could originally reach a target into a single connected component, and keeps the other connected components intact. 

\begin{proof}
	We assume an arbitrary total order on the vertices so that given a set of vertices, we can talk of the vertex with the smallest index.
	We construct $f(E')$ as follows:
	\begin{itemize}
		\item For every vertex $a \in B^* \setminus \{ \target^* \}$ such that $a \leadsto_{E'} \targets,$ we add the edge $\{a, \target^* \}$ to $f(E').$
		\item We consider the connected components $C_1, C_2, \ldots, C_s$ of the graph $(B, E').$ For every $C_i$, if $C_i \cap \targets = \emptyset$ and $C_i \cap B^* \neq \emptyset,$ we let $c_i$ be the vertex in $C_i \cap B^*$ with the smallest index. For every vertex $c'_i \in C_i \cap B^* \setminus \{c_i\}$, we add the edge $\{c_i, c'_i\}$ to $f(E')$. 
	\end{itemize}
	Figure~\ref{fig:f} shows an example application of $f$.
	
	\begin{figure}[H]
		\centering
		\includegraphics[scale=0.6]{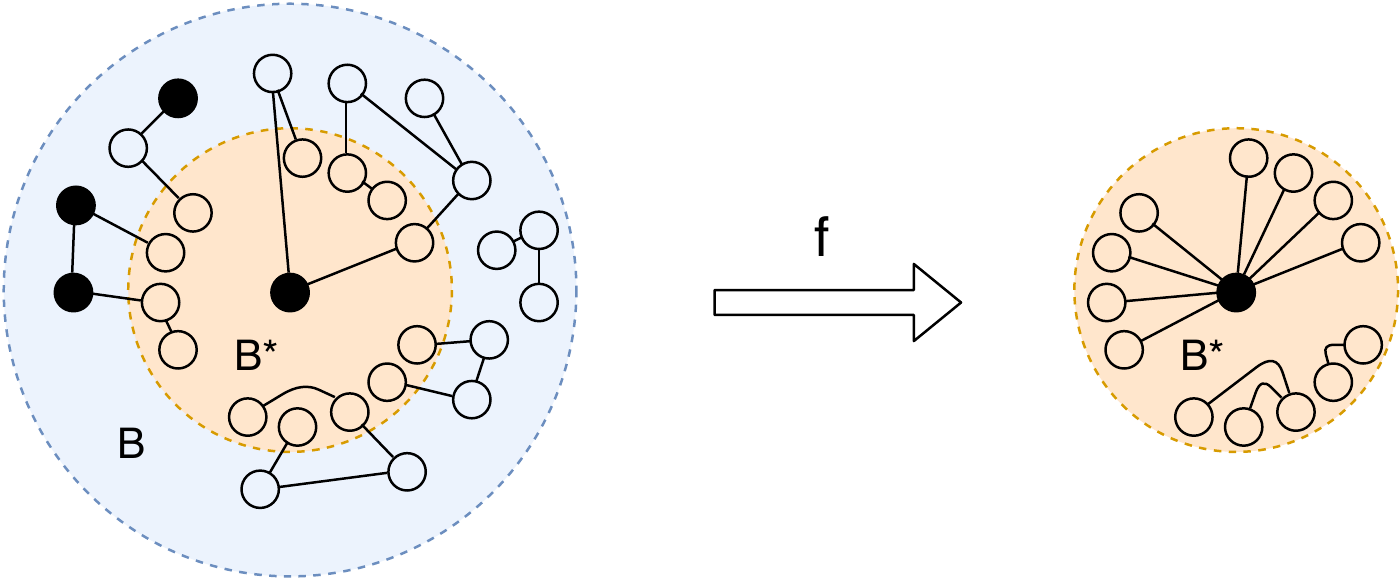}
		\caption{An example application of $f$. Target vertices are shown in black. The graph $(B, E')$ is shown on the left and $(B^*, f(E'))$ is the graph on the right.}
		\label{fig:f}
	\end{figure}
	
	By the above construction, it is easy to verify that $f(E')$ has the desired properties. Moreover, $f(E')$ can be computed by a single use of a classical reachability algorithm, such as DFS, that finds the connected components $C_i$. Hence, it can be computed in $O(\vert B \vert + \vert E_B \vert).$
\end{proof}

\begin{lemma}[Shrinking Lemma] \label{lemma:shrink}
	Let $I = (G, E_1, \ldots, E_r, \source, \targets, \Prob_1, \ldots, \Prob_r)$ be an \problem{Extended Network Reliability} instance such that:
	\begin{itemize}
		\item $G = (V, E)$ and $(A, B)$ is a separation of $G$;
		\item $\source \in A$;
		\item There exists a target vertex $\target^* \in A \cap B \cap \targets;$
		\item Every part $E_i$ is either entirely in $A$ or entirely in $B$. Without loss of generality, we assume that $E_1, \ldots, E_t$ are entirely in $A$ and $E_{t+1}, \ldots, E_r$ are entirely in $B$. We define $E_A := E_1 \cup \cdots \cup E_t$ and $E_B := E_{t+1} \cup \cdots \cup E_r$. If an $E_i$ is entirely in $A \cap B$, we consider it to be part of $E_B$;
	\end{itemize}
	then, there exists a smaller instance $\overline{I} = (\overline{G}, E_1, \ldots, E_t, \overline{E_{t+1}}, \Prob_1, \ldots, \Prob_t, \overline{\Prob_{t+1}})$ with vertex set $A$, i.e. $\overline{G} = (A, \overline{E}),$ such that $\rel(\overline{I})=\rel(I).$
\end{lemma}

Intuitively, given a few conditions, this lemma provides a way of shrinking an instance by means of decreasing the number of vertices in the instance, i.e.~removing all the vertices in $B \setminus A$, without changing the reliability. 

\begin{proof}
	We use the function $f$ as in Lemma~\ref{lemma:avvali}, by considering $B^* = A \cap B$ and $G_B = (B, E_{t+1} \cup \cdots \cup E_r) = (B, E_B)$. We let the new part $\overline{E_{t+1}} = E^*$ consist of all the possible edges over $B^*.$ We define the probability distribution $\overline{\Prob_{t+1}}: 2^{\overline{E_{t+1}}} \rightarrow [0, 1]$ as follows. For every $F \subseteq \overline{E_{t+1}},$ we let 
	\begin{equation} \label{eq:def}
		\overline{\Prob_{t+1}}(F) := \sum_{E' \in f^{-1}(F)} \prod_{i=t+1}^r \Prob_i(E' \cap E_i).
	\end{equation}
	
	\noindent\emph{Claim.} \label{claim} Let $E'_A \subseteq E_A$ and $E'_B \subseteq E_B$, we claim that $\source \leadsto_{E'_A \cup E'_B} \targets$ if and only if $\source \leadsto_{E'_A \cup f(E'_B)} \targets.$
	
	\noindent\emph{Proof of Claim.} Note that $B^* = A \cap B$ is a separator in $G$, so there is no edge between $A \setminus B$ and $B \setminus A$. Consider the graphs $G_1 = (A \cup B, E'_A \cup E'_B)$ and $G_2 = (A, E'_A \cup f(E'_B))$, we want to prove that $\source \leadsto_{G_1} \targets$ if and only if $\source \leadsto_{G_2} \targets.$ 
	
	First, we assume $s \leadsto_{G_1} \targets.$ Consider a path $\pi := u_0, u_1, \ldots, u_l$ in $G_1$ from $u_0 = \source$ to some target vertex $u_l \in \targets.$ We construct a path $\pi'$ from $\source$ to $\targets$ in $G_2$ by following $\pi$ step-by-step. At each step, we assume that we have a prefix of $\pi'$ from $\source$ to $u_i$. We extend this prefix as follows:
	\begin{enumerate}
		\item If the edge $\{u_i, u_{i+1}\}$ is in $E'_A$, then it has appeared in both $G_1$ and $G_2$. So we simply extend $\pi'$ by adding $u_{i+1}.$ In particular, this case always happens if at least one of $u_i$ and $u_{i+1}$ are in $A \setminus B$.
		\item Otherwise, if both $u_i$ and $u_{i+1}$ are in $A \cap B$, then by Lemma~\ref{lemma:avvali}, we have $u_i \leadsto_{f(E'_B)} u_{i+1}.$ So, we extend $\pi'$ by a path from $u_i$ to $u_{i+1}$ in $f(E'_B)$.
		\item Otherwise, if $u_i \in A \cap B$ and $u_{i+1} \in B \setminus A$, then we consider two cases:
		\begin{enumerate}
			\item if the path $\pi$ ends in a target vertex in $B \setminus A$ without coming back to $A \cap B$, then by Lemma~\ref{lemma:avvali}, we have $u_i \leadsto_{f(E'_B)} \target^*$. We extend $\pi'$ by adding the path from $u_i$ to $\target^*$.
			\item otherwise, $\pi$ re-enters $A \cap B$. Assume that this re-entry happens at $u_j$ for some $j>i+1$. Then, by Lemma~\ref{lemma:avvali}, we have $u_i \leadsto_{f(E'_B)} u_j$. So, we extend $\pi'$ with the path from $u_i$ to $u_j$ in $f(E'_B)$ and continue from $u_j$.
		\end{enumerate}
		
	\end{enumerate}
	Note that no other case is possible, given that $A \cap B$ is a separator. Also, following the steps above, we either end up at $\target^*$ or $u_l$, both of whom are target vertices. Therefore $\source \leadsto_{G_2} \targets.$ The other side can be proven similarly, i.e.~by taking a path in $G_2$ and replacing every contiguous sequence of edges in $f(E'_B)$ by edges in $E'_B$.

	\medskip
	\medskip
	We now continue our proof of the shrinking lemma. We prove that $\rel(\overline{I}) = \rel(I).$ We have 
	$$
	\rel(I) = \sum_{E' \subseteq E} \prob(G^s = (V, E')) \cdot \delta({\source \leadsto_{E'} \targets} )
	$$
	where $\delta$ is the indicator function that has a value of 1 if its parameter is true and 0 otherwise, i.e.~$\delta(p) = \left\{\begin{matrix}
	1 & p\\ 
	0 & \neg p
	\end{matrix}\right..$ Therefore, if we consider $E'_A = E' \cap E_A$ and $E'_B = E' \cap E_B$, we have:
	$$
	\rel(I)  = \sum_{E'_A \subseteq E_A} \sum_{E'_B \subseteq E_B} \prob(G^s = (V, E'_A \cup E'_B)) \cdot \delta({\source \leadsto_{E'_A \cup E'_B} \targets} ) 
	$$
	
	We now divide the second summand based on the value of $f(E'_B),$ to get:
	
	$$
	\rel(I) =  \sum_{E'_A \subseteq E_A} \sum_{F \subseteq \overline{E_{t+1}}} \sum_{E'_B \in f^{-1}(F)} \prob(G^s = (V, E'_A \cup E'_B)) \cdot \delta({\source \leadsto_{E'_A \cup E'_B} \targets} )\\
	$$
	
	The appearance of edges in $E_A$ and $E_B$ are independent of each other, so we have $\prob(G^s = (V, E'_A \cup E'_B)) = \prob(G^s \cap E_A = E'_A) \cdot \prob(G^s \cap E_B = E'_B),$ therefore:
	$$
	\rel(I) =  \sum_{E'_A \subseteq E_A} \sum_{F \subseteq \overline{E_{t+1}}} \sum_{E'_B \in f^{-1}(F)} \prob(G^s \cap E_A = E'_A) \cdot \prob(G^s \cap E_B = E'_B) \cdot \delta({\source \leadsto_{E'_A \cup E'_B} \targets} ) 
	$$
	
	Given that $\prob(G^s\cap E_A = E'_A)$ is independent of the two inner summations, we have:
	$$
	\rel(I) =  \sum_{E'_A \subseteq E_A} \prob(G^s \cap E_A = E'_A) \cdot \sum_{F \subseteq \overline{E_{t+1}}} \sum_{E'_B \in f^{-1}(F)}  \prob(G^s \cap E_B = E'_B) \cdot \delta({\source \leadsto_{E'_A \cup E'_B} \targets} ) 
	$$
	
	Moreover, appearance of edges in each $E_i$ is independent of every other $E_j$, so $\prob(G^s \cap E_A = E'_A) = \prod_{i=1}^t \Prob_i(E'_A \cap E_i)$ and similarly, $\prob(G^s \cap E_B = E'_B) = \prod_{i=t+1}^r \Prob_i(E'_B \cap E_i).$ Therefore, we have:
	$$
	\rel(I) = \sum_{E'_A \subseteq E_A} \prod_{i=1}^{t} \Prob_i(E'_A \cap E_i) \cdot \sum_{F \subseteq \overline{E_{t+1}}} \sum_{E'_B \in f^{-1}(F)}  \prod_{i=t+1}^r \Prob_i(E'_B \cap E_i) \cdot \delta({\source \leadsto_{E'_A \cup E'_B} \targets} )
	$$
	
	According to the Claim proven in Page~\pageref{claim}, we have $\delta(\source \leadsto_{E'_A \cup E'_B} \targets) = \delta(\source \leadsto_{E'_A \cup f(E'_B) } \targets)$. Note that $f(E'_B)$ is simply $F$, given that $E'_B \in f^{-1}(F)$. So, we have: 
	$$
	\rel(I) =  \sum_{E'_A \subseteq E_A} \prod_{i=1}^{t} \Prob_i(E'_A \cap E_i) \cdot \sum_{F \subseteq \overline{E_{t+1}}} \sum_{E'_B \in f^{-1}(F)}  \prod_{i=t+1}^r \Prob_i(E'_B \cap E_i) \cdot \delta({\source \leadsto_{E'_A \cup F} \targets} )
	$$
	
	In Equation~\eqref{eq:def}, we defined $\overline{\Prob_{t+1}}(F)$ as  $\sum_{E' \in f^{-1}(F)} \prod_{i=t+1}^r \Prob_i(E' \cap E_i)$, therefore:
	$$
	\rel(I) = \sum_{E'_A \subseteq E_A} \prod_{i=1}^{t} \Prob_i(E'_A \cap E_i) \cdot \sum_{F \subseteq \overline{E_{t+1}}} \overline{\Prob_{t+1}} (F) \cdot \delta({\source \leadsto_{E'_A \cup F} \targets} )
	$$
	
	We now reverse all the steps above, using $\overline{G^s}$, i.e.~the probabilistic graph obtained according to the \problem{Extended Network Reliability} instance $\overline{I}$:
	
	$$
	\begin{matrix*}[l]
	\rel(I) & = & \sum_{E'_A \subseteq E_A} 
	\prob(\overline{G^s} \cap E_A = E'_A) \cdot \sum_{F \subseteq \overline{E_{t+1}}} \prob(\overline{G^s} \cap \overline{E_{t+1}} = F)  \cdot \delta(\source \leadsto_{E'_A \cup F} \targets) \\
	
	& = & \sum_{E'_A \subseteq E_A} \sum_{F \subseteq \overline{E_{t+1}}} 
	\prob(\overline{G^s} \cap E_A = E'_A) \cdot \prob(\overline{G^s} \cap \overline{E_{t+1}} = F)  \cdot \delta(\source \leadsto_{E'_A \cup F} \targets) \\
	
	& = & \sum_{E'_A \subseteq E_A} \sum_{F \subseteq \overline{E_{t+1}}} 
	\prob(\overline{G^s} = (V, E'_A \cup F)) \cdot \delta(\source \leadsto_{E'_A \cup F} \targets) \\
	
	& = & \sum_{\overline{E'} \subseteq E_A \cup \overline{E_{t+1}}} 
	\prob(\overline{G^s} = (V, \overline{E'}))  \cdot \delta(\source \leadsto_{\overline{E'}} \targets) \\
	
	& = & \rel(\overline{I}).
	
	\end{matrix*}
	$$
\end{proof}

\smallskip\noindent\textbf{The Complexity of Shrinking Lemma.} To apply the Shrinking Lemma, we compute $f(E'_B)$ for every $E'_B \subseteq E_B$. There are $2^{\vert E_B \vert}$ such subsets. Moreover, by Lemma~\ref{lemma:avvali}, each computation of $f$ takes $O(\vert B \vert + \vert E_B \vert)$, therefore the overall complexity of the Shrinking Lemma is $O(2^{\vert E_B \vert} \cdot (\vert B \vert + \vert E_B \vert)),$ which is exponential in the size of $E_B$. Hence, the Shrinking Lemma should only be applied when the set $B$ is small. In our algorithm below, whenever we use the Shrinking Lemma, the set $B$ is a bag of the tree decomposition and therefore has size at most $O(k)$.

The following lemma provides the last ingredient for our main algorithm:
\begin{lemma} \label{lemma:merge}
	Let $I = (G, E_1, \ldots, E_r, \source, \targets, \Prob_1, \ldots, \Prob_r)$ be an \problem{Extended Network Reliability} instance with $r \geq 2$. There exists an instance $I_{r-1,r} = (G, E_1, \ldots, E_{r-2}, E_{r+1}, \source, \targets, \Prob_1, \ldots, \Prob_{r-2}, \Prob_{r+1})$ such that $\rel(I) = \rel(I_{r-1, r}).$ We refer to $I_{r-1, r}$ as the instance obtained by \emph{merging} $E_{r-1}$ and $E_r$ in $I$. 
\end{lemma}

\begin{proof}
	Let $W \subseteq V$ be the subset of vertices consisting of all the endpoints of edges in $E_{r-1}$ and $E_r$. We define $E_{r+1}$ as a set containing all possible edges over $W$, i.e.~$E_{r+1} = \{ \{u, v\} ~\vert~ u, v \in W \}.$ Let $F_{r-1} \subseteq E_{r-1}$ and $F_{r} \subseteq E_r$, we define $F_{r-1} \oplus F_r$ as the subset of $E_{r+1}$ that contains an edge from $u$ to $v$ if at least one of $F_{r}$ and $F_{r+1}$ do. Intuitively, $F_{r-1} \oplus F_r$ is a special kind of union that ignores repeated edges with the same endpoints. For every $F_{r+1} \subseteq E_{r+1},$ we define
	$$
	\Prob_{r+1}(F_{r+1}) := \sum_{F_{r-1} \oplus F_r = F_{r+1}} \Prob_{r-1}(F_{r-1}) \cdot \Prob_r(F_r).
	$$
	Informally, we took two parts $E_{r-1}$ and $E_{r}$ which used to be independent and merged them into a single correlated part. It is straightforward to verify that this construction preserves the reliability.
\end{proof}

Note that the order of $E_i$'s in $I$ does not matter. Hence, we can define $I_{i,j}$ as the instance obtained by merging $E_i$ and $E_j$ in $I$ and construct it in a similar manner.

We are now ready to present our main algorithm for computing \problem{Network Reliability} when the underlying graph has a small treewidth $k$. Basically, our algorithm applies the Shrinking Lemma repeatedly until the number of vertices in the graph is reduced to $O(k).$ Then, it runs the brute force algorithm over it.

\smallskip\noindent\textbf{Input.} As mentioned earlier, the input to the algorithm is a \problem{Network Reliability} instance $I = (G, \source, \targets,  \Prob)$ together with a $k$-decomposition $T = (\bags, E_T)$ of the graph~$G$ rooted at a bag $r \in \bags$ that contains the source vertex, i.e.~$\source \in V(r)$. 

\smallskip\noindent\textbf{Our Algorithm.} We compute $\rel(I)$ as follows:
\begin{enumerate}
	\item We take an arbitrary target vertex $\target^* \in \targets$ and add it to the vertex set of every bag in $\bags.$ 
	\item As long as $\bags$ contains more than one bag, we do the following:
	\begin{enumerate}
		\item We take an arbitrary leaf bag $b \in \bags$. Let $p$ be the parent bag of $b$. 
		\item We apply the Shrinking Lemma with $(A, B) = (\bigcup_{a \in \bags \setminus \{b\}} V(a) , V(b))$. This effectively removes all the vertices in $B \setminus A$ from $G$.
		\item We remove $b$ from $T$.
		\item We take all the edge parts $E_{i_1}, E_{i_2}, \ldots, E_{i_j}$ that are entirely in $V(p)$ and merge them together using Lemma~\ref{lemma:merge}.
	\end{enumerate}
	\item If $\bags$ contains a single bag $r$, we simply run the brute force algorithm on $V(r)$ for computing $\rel(I)$.
\end{enumerate}
Consider the graph depicted in Figure~\ref{fig:tw} with arbitrary probabilities. We assume $\source = 1$ and $\targets = \{ 7 \}.$ Therefore, in Step~(1), we add $7$ to every bag. Figure~\ref{fig:mainex} shows the iterations of Step~(2), i.e.~each panel shows one application of shrinking and merging. The figure does not show the probabilities. In each iteration, a leaf bag $b$ with parent $p$ is chosen, it is removed from the tree decomposition and the vertices that only appeared in $b$ are deleted from the graph. Moreover, a new edge part is added that covers all possible edges between the vertices in the intersection of $V(b)$ and $V(p).$ It is then merged with the already existing edge parts of $V(b).$ This process continues until only one bag $r$ (the root) remains. At this point, the brute force algorithm is used to compute the reliability.

\begin{figure}
	\includegraphics[keepaspectratio,width=\linewidth]{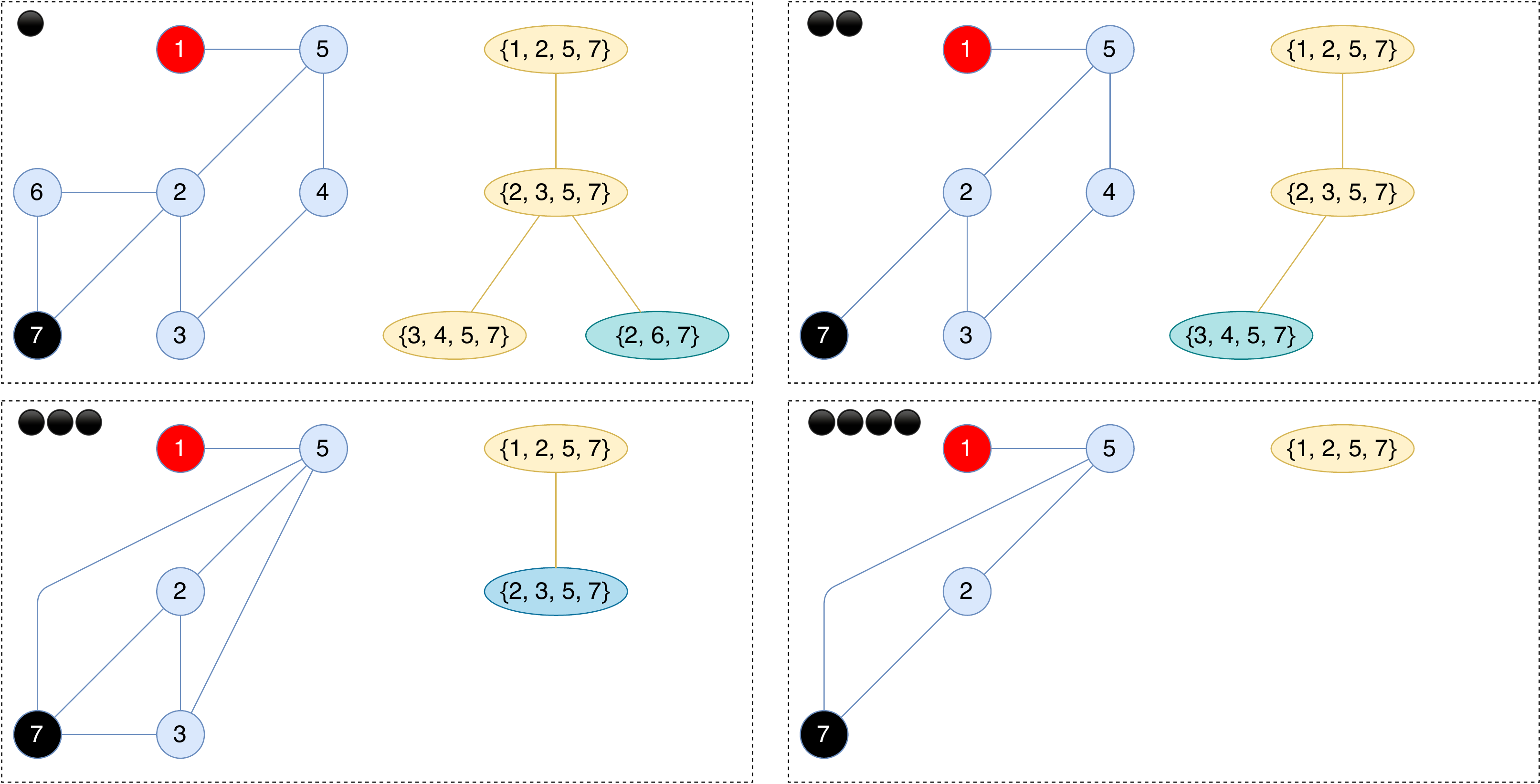}
	\caption{Changes in $G$ and $T$ in every iteration of Step~(2)}
	\label{fig:mainex}
\end{figure}

\begin{lemma}[Correctness]
	The algorithm above correctly computes $\rel(I).$
\end{lemma}

\begin{proof}
	Each iteration of Step~2 above reduces the number of bags by one, hence the algorithm terminates. We show that the reliability is preserved after each iteration of Step~2. We also show that before and after each iteration of Step~2, for every correlated edge part $E_i$, there exists a bag $b_{E_i} \in \bags$ such that every endpoint of every edge in $E_i$ is in $V(b_{E_i}).$ Note that initially, each $E_i$ consists of a single edge and hence this property holds by definition of tree decompositions.
	
	In Step~2(b), we can apply the Shrinking Lemma because (i)~$(A, B)$ is a separation of $G$ by the Cut Property (Lemma~\ref{lemma:cut}); (ii)~$\source \in V(r) \subseteq A$; (iii)~$\target^* \in A \cap B \cap \targets$ because $\target^*$ is in the vertex set of every bag; and (iv)~for every edge part $E_i$, there exists a bag such that $E_i$ is entirely in that bag. As shown in Page~\pageref{lemma:shrink}, Shrinking Lemma preserves the reliability. Moreover, by Lemma~\ref{lemma:cut}, we have $A \cap B = V(p) \cap V(b)$ and therefore the newly added edge part is entirely in $A \cap B \subseteq V(p).$ Finally, by Lemma~\ref{lemma:merge}, merging edge parts in Step~2(d) does not change the reliability.
	
	It follows by an easy induction that the reliability is preserved when we reach Step 3. At this point, $\rel(I)$ is computed by the brute force algorithm. Hence, our algorithm computes $\rel(I)$ correctly.
\end{proof}

\smallskip\noindent\textbf{Complexity of Our Algorithm.} The algorithm applies the shrinking lemma $O(n \cdot k)$ times and in each time we have $\vert B \vert \leq k+2$ and $\vert E_B \vert \leq (k+2) \cdot (k+1)$. Hence, the overall runtime for the calls to shrinking lemma is $O(n \cdot k^3 \cdot 2^{(k+2)(k+1)})$. Similarly, the algorithm performs $O(n \cdot k)$ merge operations, each of which on two parts of size at most $2^{\binom{k+2}{2}}.$ Hence each merge operation takes at most $O(k^2 \cdot 2^{(k+2)(k+1)})$ and the overall runtime for merging is also $O(n \cdot k^3 \cdot 2^{(k+2)(k+1)})$. Finally, the algorithm runs the brute force procedure on a graph with at most $\binom{k+2}{2}$ edges, which takes $O(k^2 \cdot 2^{\binom{k+2}{2}})$ time. All the other operations are performed in linear time. Hence the total runtime of our algorithm is $O(n \cdot k^3 \cdot 2^{(k+2)(k+1)}),$ which depends linearly on $n$ and exponentially on $k$. Hence, we have the following theorem:

\begin{theorem*}
	There exists a linear-time fixed-parameter algorithm for solving \problem{Network Reliability} when parameterized by the treewidth, i.e.~when the treewidth is a small constant.
\end{theorem*}

\smallskip\noindent\textbf{Pseudocode.} Our approach is summarized in Algorithm~\ref{algo:main}.

\begin{algorithm}[H]
	\SetKwInOut{Input}{Input}\SetKwInOut{Output}{Output}
	\Input{A \problem{Network Reliability} instance $I = (G, \source, \targets, \Prob)$ or equivalently an \problem{Extended Network Reliability} instance $I = (G, E_1, \ldots, E_m, \source, \targets, \Prob_1, \ldots, \Prob_m)$ in which $\vert E_i \vert = 1$ for every $i$; and a $k$-decomposition $T = (\bags, E_T)$ of $G$ rooted at $r$ with $\source \in V(r).$ 
	}
	\Output{$\rel(I)$}
	
	$\target^* \leftarrow $ an arbitrary vertex in $\targets$\;
	
	\ForEach{$b \in \bags$}
	{
		$V(b) \leftarrow V(b) \cup \{ \target^* \}$\;
	}
	
	\While{$\vert \bags \vert > 1$}
	{
		$b \leftarrow $ an arbitrary leaf bag in $\bags$\;
		$A \leftarrow \cup_{a \in \bags \setminus \{b\}} V(a)$\;
		$B \leftarrow V(b)$\;
		Shrinking-Lemma($I, A, B$)\;
		$p \leftarrow b . $parent\;
		$\bags \leftarrow \bags \setminus \{b\}$\;
		\While{$\exists E_i, E_j$ s.t. $\forall e=\{u, v\} \in E_i \cup E_j~~u, v \in V(p)$}
		{
			$I \leftarrow I_{i, j}$
		}
	}

	\textbf{return} Brute-Force($I$)\;
	
	\caption{Computing \problem{Network Reliability} using a Tree Decomposition}\label{algo:main}
\end{algorithm}

\section{Implementation and Experimental Results} \label{sec:exp}

We implemented our approach in Java. Our code is available at \url{https://ist.ac.at/~akafshda/reliability}. We used a tool called \textsc{FlowCutter}~\cite{hamann2018graph} for computing the tree decompositions. \textsc{FlowCutter} applies a state-of-the-art heuristic algorithm to find tree decompositions of small width. However, it is not guaranteed to find an optimal tree decomposition. In each case, we limited \textsc{FlowCutter} to a maximum runtime of 10 minutes.

We experimented with the subway networks of several major cities, including Berlin, London, Tehran, Tokyo and Vienna. Table~\ref{tab:results} provides a summary of the instances. Notably, it shows that all of these major subway networks have a small treewidth. In each case, we set the source and target vertices as the subway stations next to some of the major universities. Specifically:
\begin{itemize}
	\item In Berlin, we used the Technical University of Berlin (Ernst-Reuter-Platz) as the source and the Freie University (Thielplatz) and the Humboldt University (Friedrichstrasse) as the targets.
	\item In London, we used the London School of Economics (Holborn) as the source and the King's College (Temple), Imperial College (South Kensington) and University College London (Euston Square) as the targets.
	\item In Tehran, we used Amirkabir University (Teatr-e Shahr) as the source and Sharif University (Daneshgah-e Sharif), University of Tehran (Meydan-e Enqelab) and the Iran University of Science and Technology (Daneshgah-e Elm o San'at) as the targets.
	\item In Tokyo, we used the University of Tokyo (Okachimachi) as the source and Keio University (Mita) and Waseda University (Waseda) as the targets.
	\item In Vienna, we used the University of Vienna (Schottenring) as the source and the Technical University of Vienna (Karlsplatz) as the target.
\end{itemize}

In our experiments, we assumed that every edge of the network appears with the same probability $p$. We provide experimental results for different values of $p$ between $0$ and $1$ with step size $0.05$. Note that this is not a requirement of our algorithm, which can handle different probability values for different edges. However, we did not have access to the failure probabilities of the connections in the subway networks, given that such information is classified in most countries. 

\begin{table}
	\centering
	\begin{tabular}{|c|c|c|c|c|c|}
		\hline
		Network & $n$ & $\vert \targets \vert$ & $k$ & Runtime (s)\\
		\hline \hline
		Berlin U-Bahn & $173$ & $2$ & $3$ & $1.4$ \\ \hline
		London Tube & $261$ & $3$ & $5$ & $1.9$\\ \hline
		Tehran Metro & $103$ & $3$ & $2$ & $1.2$ \\ \hline
		Tokyo Subway & $200$ & $2$ & $6$ &  $12.4$\\ \hline
		Vienna U/S-Bahn & $141$ & $1$ & $5$ & $1.6$ \\
		\hline
	\end{tabular}
	\caption{A Summary of Our Benchmark Networks. In each case $n$ is the number of vertices in the network, $\vert \targets \vert$ is the number of target vertices, and $k$ is the width of the obtained tree decomposition. The runtimes are reported in seconds and are the average runtime over all $p$. }
	\label{tab:results}
\end{table}

As shown in Table~\ref{tab:results}, our algorithm is extremely efficient and, in all these real-world cases, answers the \problem{Network Reliability} problem in just a few seconds. In contrast, previous exact approaches such as~\cite{mohammadi2016divisors,saenz2015hilbert} could only handle academic examples with less than 10 vertices. Figure~\ref{fig:exp} provides a summary of our experimental results.

\begin{figure}[H]
	\centering
	\includegraphics[keepaspectratio,width=.5\linewidth]{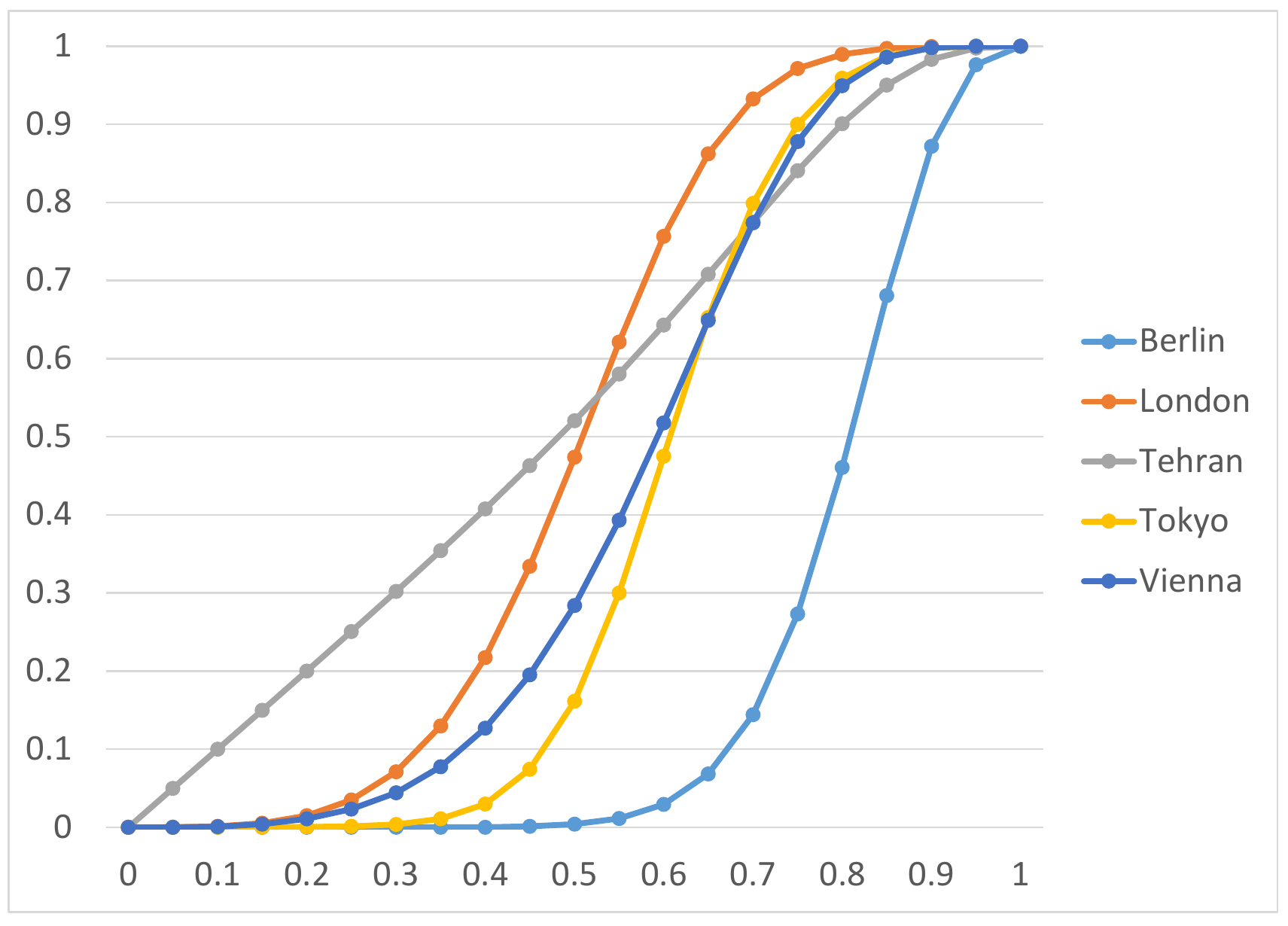}
	\caption{Our Experimental Results. The $x$-axis is the probability $p$ of the appearance of each edge of the network, and the $y$-axis is the reliability $\rel(I).$}
	\label{fig:exp}
\end{figure}

\section{Conclusion}
In this paper, on the theoretical side, we presented a linear-time algorithm for computing \problem{Network Reliability} on graphs with small treewidth. Our algorithm uses the concept of kernelization, i.e.~it repeatedly transforms an instance into a smaller one with the same reliability. On the experimental side, we showed that subway networks of several major cities have small treewidth and hence our algorithm can be applied to them. We also demonstrated that our algorithm is extremely efficient and can handle these real-world instances in a few seconds, while previous exact methods could only handle academic examples with a handful of edges.

\vspace{0.5cm}
\smallskip\noindent\textbf{Acknowledgments.}
The research was partially supported by the EPSRC Early Career Fellowship EP/R023379/1, Grant \textnumero~SC7-1718-01 of the London Mathematical Society, an IBM PhD Fellowship, and a DOC Fellowship of the Austrian Academy of Sciences (\"{O}AW).


\newpage
\section*{References}

\begin{thebibliography}{36}
	\expandafter\ifx\csname natexlab\endcsname\relax\def\natexlab#1{#1}\fi
	\providecommand{\url}[1]{\texttt{#1}}
	\providecommand{\href}[2]{#2}
	\providecommand{\path}[1]{#1}
	\providecommand{\DOIprefix}{doi:}
	\providecommand{\ArXivprefix}{arXiv:}
	\providecommand{\URLprefix}{URL: }
	\providecommand{\Pubmedprefix}{pmid:}
	\providecommand{\doi}[1]{\href{http://dx.doi.org/#1}{\path{#1}}}
	\providecommand{\Pubmed}[1]{\href{pmid:#1}{\path{#1}}}
	\providecommand{\bibinfo}[2]{#2}
	\ifx\xfnm\relax \def\xfnm[#1]{\unskip,\space#1}\fi
	\bibitem[{Agrawal and Barlow(1984)}]{reliabilitySurvey}
	\bibinfo{author}{A.~Agrawal}, \bibinfo{author}{R.~E. Barlow},
	\newblock \bibinfo{title}{A survey of network reliability and domination
		theory},
	\newblock \bibinfo{journal}{Operations Research} \bibinfo{volume}{32}
	(\bibinfo{year}{1984}) \bibinfo{pages}{478--492}.
	\bibitem[{Moore and Shannon(1956)}]{reliabilityApplication}
	\bibinfo{author}{E.~F. Moore}, \bibinfo{author}{C.~E. Shannon},
	\newblock \bibinfo{title}{Reliable circuits using less reliable relays},
	\newblock \bibinfo{journal}{Journal of the Franklin Institute}
	\bibinfo{volume}{262} (\bibinfo{year}{1956}) \bibinfo{pages}{191--208}.
	\bibitem[{Ball(1986)}]{ball1986computational}
	\bibinfo{author}{M.~O. Ball},
	\newblock \bibinfo{title}{Computational complexity of network reliability
		analysis: An overview},
	\newblock \bibinfo{journal}{IEEE Transactions on Reliability}
	\bibinfo{volume}{35} (\bibinfo{year}{1986}) \bibinfo{pages}{230--239}.
	\bibitem[{Satyanarayana and Wood(1985)}]{satyanarayana1985linear}
	\bibinfo{author}{A.~Satyanarayana}, \bibinfo{author}{R.~K. Wood},
	\newblock \bibinfo{title}{A linear-time algorithm for computing k-terminal
		reliability in series-parallel networks},
	\newblock \bibinfo{journal}{SIAM Journal on Computing} \bibinfo{volume}{14}
	(\bibinfo{year}{1985}) \bibinfo{pages}{818--832}.
	\bibitem[{Provan and Ball(1984)}]{provan1984computing}
	\bibinfo{author}{J.~S. Provan}, \bibinfo{author}{M.~O. Ball},
	\newblock \bibinfo{title}{Computing network reliability in time polynomial in
		the number of cuts},
	\newblock \bibinfo{journal}{Operations Research} \bibinfo{volume}{32}
	(\bibinfo{year}{1984}) \bibinfo{pages}{516--526}.
	\bibitem[{Coit and Smith(1996)}]{coit1996reliability}
	\bibinfo{author}{D.~W. Coit}, \bibinfo{author}{A.~E. Smith},
	\newblock \bibinfo{title}{Reliability optimization of series-parallel systems
		using a genetic algorithm},
	\newblock \bibinfo{journal}{IEEE Transactions on reliability}
	\bibinfo{volume}{45} (\bibinfo{year}{1996}) \bibinfo{pages}{254--260}.
	\bibitem[{Karger(2001)}]{karger2001randomized}
	\bibinfo{author}{D.~R. Karger},
	\newblock \bibinfo{title}{A randomized fully polynomial time approximation
		scheme for the all-terminal network reliability problem},
	\newblock \bibinfo{journal}{SIAM review} \bibinfo{volume}{43}
	(\bibinfo{year}{2001}) \bibinfo{pages}{499--522}.
	\bibitem[{Srivaree-ratana et~al.(2002)Srivaree-ratana, Konak, and
		Smith}]{srivaree2002estimation}
	\bibinfo{author}{C.~Srivaree-ratana}, \bibinfo{author}{A.~Konak},
	\bibinfo{author}{A.~E. Smith},
	\newblock \bibinfo{title}{Estimation of all-terminal network reliability using
		an artificial neural network},
	\newblock \bibinfo{journal}{Computers \& Operations Research}
	\bibinfo{volume}{29} (\bibinfo{year}{2002}) \bibinfo{pages}{849--868}.
	\bibitem[{Gertsbakh and Shpungin(2016)}]{gertsbakh2016models}
	\bibinfo{author}{I.~B. Gertsbakh}, \bibinfo{author}{Y.~Shpungin},
	\bibinfo{title}{Models of network reliability: analysis, combinatorics, and
		Monte Carlo}, \bibinfo{publisher}{CRC press}, \bibinfo{year}{2016}.
	\bibitem[{Brown et~al.(1996)Brown, Colbourn, and Wagner}]{brown1996cohen}
	\bibinfo{author}{J.~I. Brown}, \bibinfo{author}{C.~J. Colbourn},
	\bibinfo{author}{D.~G. Wagner},
	\newblock \bibinfo{title}{Cohen--macaulay rings in network reliability},
	\newblock \bibinfo{journal}{SIAM Journal on Discrete Mathematics}
	\bibinfo{volume}{9} (\bibinfo{year}{1996}) \bibinfo{pages}{377--392}.
	\bibitem[{Mohammadi(2016)}]{fm2016combinatorial}
	\bibinfo{author}{F.~Mohammadi},
	\newblock \bibinfo{title}{Combinatorial and geometric view of the system
		reliability theory},
	\newblock in: \bibinfo{booktitle}{International Congress on Mathematical
		Software}, \bibinfo{organization}{Springer}, \bibinfo{year}{2016}, pp.
	\bibinfo{pages}{148--153}.
	\bibitem[{S{\'a}enz-de Cabez{\'o}n and Wynn(2015)}]{saenz2015hilbert}
	\bibinfo{author}{E.~S{\'a}enz-de Cabez{\'o}n}, \bibinfo{author}{H.~P. Wynn},
	\newblock \bibinfo{title}{Hilbert functions in design for reliability},
	\newblock \bibinfo{journal}{IEEE transactions on reliability}
	\bibinfo{volume}{64} (\bibinfo{year}{2015}) \bibinfo{pages}{83--93}.
	\bibitem[{Mohammadi et~al.(2016)Mohammadi, S{\'a}enz-de Cabez{\'o}n, and
		Wynn}]{mohammadi2016algebraic}
	\bibinfo{author}{F.~Mohammadi}, \bibinfo{author}{E.~S{\'a}enz-de Cabez{\'o}n},
	\bibinfo{author}{H.~P. Wynn},
	\newblock \bibinfo{title}{The algebraic method in tree percolation},
	\newblock \bibinfo{journal}{SIAM Journal on Discrete Mathematics}
	\bibinfo{volume}{30} (\bibinfo{year}{2016}) \bibinfo{pages}{1193--1212}.
	\bibitem[{Mohammadi(2016)}]{mohammadi2016divisors}
	\bibinfo{author}{F.~Mohammadi},
	\newblock \bibinfo{title}{Divisors on graphs, orientations, syzygies, and
		system reliability},
	\newblock \bibinfo{journal}{Journal of Algebraic Combinatorics}
	\bibinfo{volume}{43} (\bibinfo{year}{2016}) \bibinfo{pages}{465--483}.
	\bibitem[{Guidotti et~al.(2017)Guidotti, Gardoni, and
		Chen}]{guidotti2017network}
	\bibinfo{author}{R.~Guidotti}, \bibinfo{author}{P.~Gardoni},
	\bibinfo{author}{Y.~Chen},
	\newblock \bibinfo{title}{Network reliability analysis with link and nodal
		weights and auxiliary nodes},
	\newblock \bibinfo{journal}{Structural Safety} \bibinfo{volume}{65}
	(\bibinfo{year}{2017}) \bibinfo{pages}{12--26}.
	\bibitem[{Zhang and Mahadevan(2017)}]{zhang2017game}
	\bibinfo{author}{X.~Zhang}, \bibinfo{author}{S.~Mahadevan},
	\newblock \bibinfo{title}{A game theoretic approach to network reliability
		assessment},
	\newblock \bibinfo{journal}{IEEE Transactions on Reliability}
	(\bibinfo{year}{2017}).
	\bibitem[{Yeh et~al.(2015)Yeh, Bae, and Huang}]{yeh2015new}
	\bibinfo{author}{W.-C. Yeh}, \bibinfo{author}{C.~Bae}, \bibinfo{author}{C.-L.
		Huang},
	\newblock \bibinfo{title}{A new cut-based algorithm for the multi-state flow
		network reliability problem},
	\newblock \bibinfo{journal}{Reliability Engineering \& System Safety}
	\bibinfo{volume}{136} (\bibinfo{year}{2015}) \bibinfo{pages}{1--7}.
	\bibitem[{Gutjahr et~al.(1996)Gutjahr, Pflug, and
		Ruszczy{\'n}ski}]{gutjahr1996configurations}
	\bibinfo{author}{W.~J. Gutjahr}, \bibinfo{author}{G.~C. Pflug},
	\bibinfo{author}{A.~Ruszczy{\'n}ski},
	\newblock \bibinfo{title}{Configurations of series-parallel networks with
		maximum reliability},
	\newblock \bibinfo{journal}{Microelectronics Reliability} \bibinfo{volume}{36}
	(\bibinfo{year}{1996}) \bibinfo{pages}{247--253}.
	\bibitem[{Niedermeier(2002)}]{niedermeier2002invitation}
	\bibinfo{author}{R.~Niedermeier},
	\newblock \bibinfo{title}{Invitation to fixed-parameter algorithms}
	(\bibinfo{year}{2002}).
	\bibitem[{Cygan et~al.(2015)Cygan, Fomin, Kowalik, Lokshtanov, Marx, Pilipczuk,
		Pilipczuk, and Saurabh}]{cygan2015parameterized}
	\bibinfo{author}{M.~Cygan}, \bibinfo{author}{F.~V. Fomin},
	\bibinfo{author}{{\L}.~Kowalik}, \bibinfo{author}{D.~Lokshtanov},
	\bibinfo{author}{D.~Marx}, \bibinfo{author}{M.~Pilipczuk},
	\bibinfo{author}{M.~Pilipczuk}, \bibinfo{author}{S.~Saurabh},
	\bibinfo{title}{Parameterized algorithms}, \bibinfo{publisher}{Springer},
	\bibinfo{year}{2015}.
	\bibitem[{Downey and Fellows(2012)}]{downey2012parameterized}
	\bibinfo{author}{R.~G. Downey}, \bibinfo{author}{M.~R. Fellows},
	\bibinfo{title}{Parameterized complexity}, \bibinfo{publisher}{Springer
		Science \& Business Media}, \bibinfo{year}{2012}.
	\bibitem[{Robertson and Seymour(1990)}]{robertson1990graph}
	\bibinfo{author}{N.~Robertson}, \bibinfo{author}{P.~D. Seymour},
	\newblock \bibinfo{title}{Graph minors. iv. tree-width and
		well-quasi-ordering},
	\newblock \bibinfo{journal}{Journal of Combinatorial Theory, Series B}
	\bibinfo{volume}{48} (\bibinfo{year}{1990}) \bibinfo{pages}{227--254}.
	\bibitem[{Courcelle(1990)}]{courcelle1990monadic}
	\bibinfo{author}{B.~Courcelle},
	\newblock \bibinfo{title}{The monadic second-order logic of graphs. i.
		recognizable sets of finite graphs},
	\newblock \bibinfo{journal}{Information and computation} \bibinfo{volume}{85}
	(\bibinfo{year}{1990}) \bibinfo{pages}{12--75}.
	\bibitem[{Chatterjee et~al.(2016)Chatterjee, Goharshady, Ibsen-Jensen, and
		Pavlogiannis}]{chatterjee2016algorithms}
	\bibinfo{author}{K.~Chatterjee}, \bibinfo{author}{A.~K. Goharshady},
	\bibinfo{author}{R.~Ibsen-Jensen}, \bibinfo{author}{A.~Pavlogiannis},
	\newblock \bibinfo{title}{Algorithms for algebraic path properties in
		concurrent systems of constant treewidth components},
	\newblock in: \bibinfo{booktitle}{ACM SIGPLAN Notices},
	volume~\bibinfo{volume}{51}, \bibinfo{organization}{ACM},
	\bibinfo{year}{2016}, pp. \bibinfo{pages}{733--747}.
	\bibitem[{Courcelle and Mosbah(1993)}]{courcelle1993monadic}
	\bibinfo{author}{B.~Courcelle}, \bibinfo{author}{M.~Mosbah},
	\newblock \bibinfo{title}{Monadic second-order evaluations on tree-decomposable
		graphs},
	\newblock \bibinfo{journal}{Theoretical Computer Science} \bibinfo{volume}{109}
	(\bibinfo{year}{1993}) \bibinfo{pages}{49--82}.
	\bibitem[{Chatterjee et~al.(2019)Chatterjee, Goharshady, Okati, and
		Pavlogiannis}]{datapacking}
	\bibinfo{author}{K.~Chatterjee}, \bibinfo{author}{A.~K. Goharshady},
	\bibinfo{author}{N.~Okati}, \bibinfo{author}{A.~Pavlogiannis},
	\newblock \bibinfo{title}{Efficient parameterized algorithms for data packing},
	\newblock \bibinfo{journal}{Proceedings of the ACM on Programming Languages}
	\bibinfo{volume}{3} (\bibinfo{year}{2019}) \bibinfo{pages}{53}.
	\bibitem[{Bodlaender(1988)}]{bodlaender1988dynamic}
	\bibinfo{author}{H.~L. Bodlaender},
	\newblock \bibinfo{title}{Dynamic programming on graphs with bounded
		treewidth},
	\newblock in: \bibinfo{booktitle}{International Colloquium on Automata,
		Languages, and Programming}, \bibinfo{organization}{Springer},
	\bibinfo{year}{1988}, pp. \bibinfo{pages}{105--118}.
	\bibitem[{Chatterjee et~al.(2018)Chatterjee, Ibsen-Jensen, Goharshady, and
		Pavlogiannis}]{chatterjee2018algorithms}
	\bibinfo{author}{K.~Chatterjee}, \bibinfo{author}{R.~Ibsen-Jensen},
	\bibinfo{author}{A.~K. Goharshady}, \bibinfo{author}{A.~Pavlogiannis},
	\newblock \bibinfo{title}{Algorithms for algebraic path properties in
		concurrent systems of constant treewidth components},
	\newblock \bibinfo{journal}{ACM Transactions on Programming Languages and
		Systems (TOPLAS)} \bibinfo{volume}{40} (\bibinfo{year}{2018})
	\bibinfo{pages}{9}.
	\bibitem[{Chatterjee et~al.(2017)Chatterjee, Goharshady, Ibsen-Jensen, and
		Pavlogiannis}]{chatterjee2016Jtdec}
	\bibinfo{author}{K.~Chatterjee}, \bibinfo{author}{A.~K. Goharshady},
	\bibinfo{author}{R.~Ibsen-Jensen}, \bibinfo{author}{A.~Pavlogiannis},
	\newblock \bibinfo{title}{{JTD}ec: A tool for tree decompositions in soot},
	\newblock in: \bibinfo{booktitle}{Fifteenth International Symposium on
		Automated Technology for Verification and Analysis, ATVA},
	\bibinfo{publisher}{Springer}, \bibinfo{year}{2017}.
	\bibitem[{Wolle(2002)}]{wolle2002framework}
	\bibinfo{author}{T.~Wolle},
	\newblock \bibinfo{title}{A framework for network reliability problems on
		graphs of bounded treewidth},
	\newblock \bibinfo{journal}{Lecture notes in computer science}
	(\bibinfo{year}{2002}) \bibinfo{pages}{137--149}.
	\bibitem[{Bodlaender(1994)}]{bodlaender1994tourist}
	\bibinfo{author}{H.~L. Bodlaender},
	\newblock \bibinfo{title}{A tourist guide through treewidth},
	\newblock \bibinfo{journal}{Acta cybernetica} \bibinfo{volume}{11}
	(\bibinfo{year}{1994}).
	\bibitem[{Thorup(1998)}]{thorup1998all}
	\bibinfo{author}{M.~Thorup},
	\newblock \bibinfo{title}{All structured programs have small tree width and
		good register allocation},
	\newblock \bibinfo{journal}{Information and Computation} \bibinfo{volume}{142}
	(\bibinfo{year}{1998}) \bibinfo{pages}{159--181}.
	\bibitem[{Gustedt et~al.(2002)Gustedt, M{\ae}hle, and
		Telle}]{gustedt2002treewidth}
	\bibinfo{author}{J.~Gustedt}, \bibinfo{author}{O.~M{\ae}hle},
	\bibinfo{author}{J.~Telle},
	\newblock \bibinfo{title}{The treewidth of java programs},
	\newblock \bibinfo{journal}{Algorithm Engineering and Experiments}
	(\bibinfo{year}{2002}) \bibinfo{pages}{57--59}.
	\bibitem[{Chatterjee et~al.(2019)Chatterjee, Goharshady, and
		Goharshady}]{amirtr}
	\bibinfo{author}{K.~Chatterjee}, \bibinfo{author}{A.~K. Goharshady},
	\bibinfo{author}{E.~K. Goharshady},
	\newblock \bibinfo{title}{The treewidth of smart contracts},
	\newblock in: \bibinfo{booktitle}{ACM Symposium on Applied Computing (SAC)},
	\bibinfo{year}{2019}.
	\bibitem[{Bodlaender(1996)}]{bodlaender1996linear}
	\bibinfo{author}{H.~L. Bodlaender},
	\newblock \bibinfo{title}{A linear-time algorithm for finding
		tree-decompositions of small treewidth},
	\newblock \bibinfo{journal}{SIAM Journal on computing} \bibinfo{volume}{25}
	(\bibinfo{year}{1996}) \bibinfo{pages}{1305--1317}.
	\bibitem[{Hamann and Strasser(2018)}]{hamann2018graph}
	\bibinfo{author}{M.~Hamann}, \bibinfo{author}{B.~Strasser},
	\newblock \bibinfo{title}{Graph bisection with pareto optimization},
	\newblock \bibinfo{journal}{Journal of Experimental Algorithmics (JEA)}
	\bibinfo{volume}{23} (\bibinfo{year}{2018}).
	
\end{thebibliography}
\end{document}